\documentclass{artjlt}

\usepackage{amsmath,amsfonts,amssymb}
\usepackage{amscd,graphics,lineno}
\usepackage{mathrsfs} 
\usepackage{multirow}

\def\norm#1{\left\Vert#1\right\Vert}

\def\R{{\mathbb R}}

\def\e{\varepsilon}
\def\Homeo{{\mbox{\rm Homeo}\,}}

\def\sB{\mathcal B}
\def\sF{\mathcal F}

\def\sT{\mathcal T}
\def\sbs{\subset}
\def\ti{\times}
\def\obr{^{-1}}

\def\Diff{{\mbox{\rm Diff}\,}}
\def\Mor{{\mbox{\rm Mor}\,}}

\def\TopGr{{\bf TG}}

\def\LCG{{\bf LCLG}}

\def\Rar{\Rightarrow}

\def\om{\omega}
\def\Om{\Omega}
\def\si{\sigma}

\def\bK{{\bf K}}
\def\frf{\mathfrak f}
\def\frg{\mathfrak g}

\def\frs{\mathfrak s}
\def\frsl{\mathfrak{sl}}

\title{On epimorphisms in some categories of infinite-dimensional Lie groups}                                     
\author{Vladimir G. Pestov and Vladimir V. Uspenskij}                 
\lastname{Pestov and Uspenskij}  



\address{%
Vladimir G. Pestov\\               
Departamento de Matem\'atica\\
 Universidade Federal de Santa Catarina\\
 Trindade\\
 Florian\'opolis, SC, 
 88.040-900\\
 Brazil\\  
{\em and} \\
Dept of Mathematics and Statistics\\ 
University of Ottawa\\
 STEM Complex\\
 150 Louis-Pasteur Pvt.\\
 Ottawa, Ontario
 K1N 6N5\\
 Canada   \\      
vladimir.pestov@uottawa.ca               
}
\address{Vladimir V. Uspenskij \\
Department of Mathematics\\
 321 Morton Hall\\ 
Ohio University\\ 
Athens, Ohio 45701 \\
 USA \\
uspenski@ohio.edu
}

\begin{document}


\maketitle

\begin{abstract}
Let $X$ be a smooth compact connected manifold.
Let $G=\Diff X$ be the group of diffeomorphisms of $X$,
equipped with the $C^\infty$-topology, 
and let $H$ be the stabilizer of some point in $X$.
Then the inclusion $H\to G$, which is a morphism of two regular Fr\'echet--Lie groups, is an epimorphism in the category of smooth Lie groups modelled on complete locally convex spaces. At the same time, in the latter category, epimorphisms between finite dimensional Lie groups have dense range.
We also prove that if $G$ is a Banach--Lie group and $H$ is a proper closed subgroup, 
the inclusion $H\to G$ is not an epimorphism in the category of Hausdorff topological groups.   
\end{abstract}

\section{Introduction}
A morphism $f:X\to Y$ in a category \bK\ is an
{\it epimorphism\/} if for any object $Z$ the mapping $f^*:\Mor(Y,Z)\to
\Mor(X,Z)$, defined by $f^*(g)=gf$, is injective.

It is known that epimorphisms in the category of Hausdorff topological
groups need not have dense range \cite{U1,U2,U3}. This 
result had in its time answered a long-standing open problem, see \cite{Nu} for a survey of the previous state of knowledge. It was already known by that time that in the category of locally compact groups epimorphisms between finite dimensional Lie groups need not have dense range (Kallman, unpublished, see \cite{Nu}, Remark 2.10), however, epimorphisms between locally compact groups in the wider category of all topological groups must have a dense range (\cite{Nu}, Thm. 2.9).

Those results prompt us to investigate intermediate categories of --- possibly infinite-dimensional --- Lie groups, more precisely, group objects in the category of $C^\infty$ manifolds modelled on complete locally convex spaces (where we follow the approach by Neeb \cite{Ne}). 
We call such groups {\em locally convex Lie groups} 
and denote the category of such groups by \LCG.
We show that within this, very wide, category there exist epimorphisms even between regular 
Fr\'echet--Lie groups without dense range. 
A {\em Fr\'echet--Lie group}, or an {\em FL-group}, 
is a locally convex group modelled on a Fr\'echet space
(= a complete metrizable locally convex space).
At the same time, we strengthen Nummela's result above by observing that already in the category of regular Fr\'echet--Lie groups (viewed generally as the best-behaved class after that of Banach--Lie groups, see \cite{O}) an epimorphism between connected finite dimensional Lie groups must have a dense range. 

If \bK\ is the category
of sets, or of groups, or of abelian groups, then epimorphisms in \bK\
are precisely surjective morphisms. If \bK\ is the category of abelian
Hausdorff topological groups or of Hausdorff topological vector spaces,
then a morphism $f:X\to Y$ is an epimorphism if and only if $f(X)$ is dense
in $Y$. K.H.~Hofmann asked in the 1960s whether the same is true for the
category
\TopGr\
of all Hausdorff topological groups (not necessarily abelian).
The answer is negative. For example, let $X$ be a compact connected
topological
manifold, let $G=\Homeo X$ be the group of all self-homeo\-mor\-phisms of
$X$, equipped with the compact-open topology, and let $H=\{f\in G:
f(p)=p\}$ be the stability subgroup at some point $p\in X$. Then
the inclusion $H\to G$ is an epimorphism in \TopGr\ \cite{U2}.

Now let $X$ be a compact smooth manifold, and let $G= \Diff X$ be the group
of all diffeomorphisms of $X$. Then $G$ has a natural structure of
a regular Fr\'echet--Lie group (where regularity implies in particular the existence of an exponential map, although we will never use it), see \cite{M,H}; \cite{O}, Ch. VI, \S 2; \cite{Ne}, Ex. II.3.14. The corresponding topology on $G$, the $C^\infty$-topology, is the topology of uniform convergence of all derivatives. Let $p\in X$, and let $H=\{f\in G: f(p)=p\}$ be the stability subgroup at $p$. The group
$H$ also has a natural regular Fr\'echet--Lie group structure, see \cite{O}, p. 145, Th. 4.5. In view
of the result cited above, it is natural to ask if the inclusion $H\to G$
is an epimorphism in the category of the regular Fr\'echet--Lie groups, or indeed in the wider category of all locally convex Lie groups. Our main result answers this question in the positive.

\begin{theorem}
\label{t:1}
Let $X$ be a compact connected
smooth manifold, let $G=\Diff X$,
and let $H\sbs G$ be the stability subgroup at some point $p\in X$. Then
the inclusion $H\to G$ is an epimorphism in the category \LCG\ 
of Lie groups modelled on complete locally convex spaces.
\end{theorem}

Since $H$ is a proper closed subgroup of $G$, it follows that epimorphisms
in the category of locally convex Lie groups need not have dense range, even between regular Fr\'echet--Lie groups. 

At the same time, the inclusion $H\to G$ of Theorem~\ref{t:1}
is not an epimorphism in the category \TopGr\ (Section~\ref{s:2}, Remark~2). In other words, any two Lie group morphisms $f,g: G\to K$ to a locally convex Lie group $K$ that agree on $H$ must be equal, but there exist distinct \TopGr-morphisms $f,g:G\to K$ to a Hausdorff
topological group $K$ that agree on $H$.

In Section 2 we prove a criterion for a morphism to be an
epimorphism in \TopGr\ (Theorem~\ref{t:2.1}) and observe that Theorem~\ref{t:1}
follows from the following assertion:

\smallskip

{\em $(*)$ Use the notation of Theorem~\ref{t:1}. Every $G$-equivariant smooth map
from $X$ to a locally convex Lie $G$-group is constant}.

\smallskip 
The notion of a locally convex $G$-group is defined in the beginning of Section~2. It is a locally convex group equipped with a smooth action of $G$.
We prove the assertion $(*)$ in Section~3.

We also note that if $f\colon G\to H$ is a morphism between finite dimensional Lie groups which is an epimorphism in the category of locally convex groups (or already in the smaller category of regular Fr\'echet--Lie groups, if $G$ is connected), then $f$ must have dense range in $H$. Another new result states that any morphism $f\colon G\to H$ between Banach--Lie groups that is an epimorphism in the category \TopGr\ of Hausdorff topological groups must have dense range. Even more, if $H$ is a proper closed subgroup (not necessarily a Banach--Lie subgroup) of a Banach--Lie group $G$, then the inclusion $H\to G$ is not an epimorphism of Hausdorff topological groups. 

Finally, the anonymous referee has kindly provided an example of an embedding of connected finite dimensinal Lie groups (the subgroup of upper triangular matrices in $SL_2(\R)$) that is an epimorphism of Banach--Lie groups.


A number of unanswered questions are collected at the end of the article.

\section 
{Criterion for an inclusion to be an epimorphism}
\label{s:2}
Let $G$ and $F$ be two topological groups. Let us say that $F$ is a
{\it $G$-group\/} if $G$ acts continuously on $F$ by automorphisms, that
is, a continuous mapping $\sigma: G\times F\to F$ is given
such that for every $g\in G$ the
mapping $\si_g:F\to F$ defined by $\si_g(x)=\si(g,x)$ is an automorphism
of $F$ and the mapping $g\mapsto \si_g$ is a group homomorphism. In this
situation, the semidirect product $G\ltimes_\si F$ can be defined, which is again
a topological group. We shall usually write simply $gx$ instead of
$\si(g,x)$.

Similarly, if $G$ and $F$ are locally convex Lie groups (in the sense of \cite{Ne}), we say that $F$ is a locally convex $G$-group if, as above, an action $\si:G\times F\to F$ is given which is in addition a $C^\infty$-mapping between manifolds. In this case, the semidirect product $G\ltimes_\si F$ has a natural locally convex group Lie group structure. If both $G$ and $F$ are Fr\'echet--Lie groups, so is the semidirect product.

For a subset $A$ of a group $G$ we denote by $C_G(A)$ the centralizer of
$A$ in $G$, that is,
the subgroup $\{x\in G: xa=ax \text{ for every } a\in A\}$.

\begin{theorem}
\label{t:2.1}
Let $i:H\to G$ be a morphism in the category \TopGr\ of
Hausdorff topological groups.
The following are equivalent:
\begin{enumerate}
\item $i$ is an epimorphism in \TopGr;
\item $C_K(fi(H))=C_K(f(G))$ for any morphism $f:G\to K$ in \TopGr;
\item if $F$ is a topological $G$-group and $x\in F$ is $i(H)$-fixed
(that is, $hx=x$ for every $h\in i(H)$), then $x$ is $G$-fixed.
\end{enumerate}
If $H$ is a closed subgroup of $G$ and $i:H\to G$ is the embedding,
these three conditions are also equivalent to the fourth:\newline
(4) for any topological $G$-group $F$ any $G$-equivariant
continuous mapping $f:G/H\to F$ is constant.
\end{theorem}

\begin{proof}
Clearly a morphism $i:H\to G$ is an epimorphism in \TopGr\ if and only if
the inclusion $i(H)\to G$ is. Hence without loss of generality we may
assume that $H$ is a subgroup of $G$ and $i:H\to G$ is the inclusion.

We show that not (1) $\Rar$ not (2).
By definition,
the inclusion $H\to G$ is not an epimorphism if and only if there exist
a Hausdorff topological group $L$ and two distinct morphisms
$g,h:G\to L$
which agree on $H$. Let $C=\{1,a\}$ be
a cyclic group of order 2. The group $C$ acts on $L\times L$ by permutations
of coordinates, let $K=C\ltimes (L\times L)$ be the corresponding semidirect product.
Let $j:L\times L\to K$ be the natural embedding, and let 
$b\in K$ be the image of $a\in C$ under the natural embedding $C\to K$.
If $x,y\in L$, then $j(x,y)$ commutes with $b$ if and only if $x=y$.
Define a morphism $f:G\to K$ by $f(x)=j(g(x), h(x))$. Then $b$ commutes with
$f(H)$ but not with $f(G)$. Thus $C_K(f(H))\ne C_K(f(G))$.

We show that not (2) $\Rar$ not (1). Suppose $f:G\to K$ is a morphism
such that some $b\in K$ commutes with $f(H)$ but not with $f(G)$. Let
$\si$ be the inner automorphism of $K$ corresponding to $b$. Then $f$ and
$\si f$ are distinct morphisms of $G$ into $K$ which agree on $H$. Thus
the inclusion of $H$ into $G$ is not an epimorphism.

We show that (2) $\Rar$ (3). Let $F$ be a $G$-group, and let $x\in F$ be
an $H$-fixed element. Let $K$ be the semidirect product of $G$ and $F$,
and let $i:G\to K$ and $j:F\to K$ be the canonical embeddings.
If the condition (2) holds, then $C_K(i(H))=C_K(i(G))$.
Since $x$ is $H$-fixed, we have $j(x)\in C_K(i(H))$.
It follows that $x\in C_K(i(G))$. Hence $x$ is $G$-fixed.

We show that (3) $\Rar$ (2). Let $f:G\to K$ be a continuous homomorphism,
and let $x\in C_K(f(H))$. The group $K$ can be considered as a $G$-group,
equipped with the $G$-action defined by $(g,k)\mapsto f(g)kf(g)^{-1}$. The
element $x$ is $H$-fixed. If the condition (3) holds, then $x$ is also
$G$-fixed, which means that $x\in C_K(f(G))$.

Finally, the equivalence (3) $\Leftrightarrow$ (4) is clear, since for
any $G$-group $F$ there is a natural one-to-one correspondence between
$H$-fixed elements of $F$ and $G$-equivariant morphisms $f:G/H\to F$, which
assigns to every $H$-fixed element $x\in F$ the morphism
$gH\mapsto gx$ of $G/H$ to $F$.
\end{proof}

\begin{theorem}
\label{t:2.2}
Let $i:H\to G$ be a morphism in the category \LCG\ of
locally convex Lie groups.
The following are equivalent:
\begin{enumerate}
\item $i$ is an epimorphism in \LCG;
\item $C_K(fi(H))=C_K(f(G))$ for any morphism $f:G\to K$ in \LCG;
\item if $F$ is a locally convex Lie $G$-group and $x\in F$ is $i(H)$-fixed
(that is, $hx=x$ for every $h\in i(H)$), then $x$ is $G$-fixed.
\end{enumerate}
\end{theorem}

The proof is quite similar to the proof of Theorem~2.1 and hence omitted.
The essential point is that semidirect products exist in \LCG. We did not
include the analogue of condition (4) of Theorem~\ref{t:2.1}
in order to avoid possible ambiguities 
of the notion ``locally convex Lie subgroup of a locally convex Lie group".
However, if $X$ is a compact smooth manifold, $G=\Diff X$ and $H$ is the stabilizer
of a point $p\in X$, as in Theorem~\ref{t:1}, the quotient $G/H$ is well-defined
in the category of manifolds modelled on (complete) locally convex spaces and can be identified with $X$. This means
that for every manifold $M$ modelled on a (complete) locally convex space there is a one-to-one correspondence between
smooth maps $X\to M$ and those smooth maps $G\to M$ that are constant on cosets $gH$.
Thus we can apply the last paragraph of the proof of Theorem~\ref{t:1} (the equivalence
(3) $\Leftrightarrow$ (4)) in this situation, and we conclude that the fact that 
the embedding $H\to G$ is an epimorphism in the category \LCG\ is equivalent to the following:

\smallskip
{\em (*) If $X$ is a compact connected smooth manifold, $G=\Diff X$, and
$F$ is a locally convex Lie $G$-group, then every $G$-equivariant smooth mapping
$j:X\to F$ is constant.}
\smallskip

\noindent{{\bf Remarks.}}
1. Let $X$ be a compact connected manifold and $G=\Homeo(X)$. Let $H\sbs G$
be the stability subgroup at some point of $X$. Then the inclusion
$H\to G$ is an epimorphism in \TopGr\ \cite{U2}. In virtue of Theorem~2.1,
this assertion is equivalent to the following:

\smallskip

{\it For any topological $G$-group $F$ any $G$-equivariant mapping
$j: X\to F$ is constant}.

\smallskip

The latter was proved in \cite[Example 3.7]{Me} (for the case $X$ is a cube or a sphere,
but the general case can be proved by the same method, Theorem~3.5 of \cite{Me} applies).
Thus, as noted in \cite{Pe, P2}, the solution of the epimorphism problem for Hausdorff topological
groups obtained in \cite{U1, U2} can be deduced from Megrelishvili's
results.

We sketch the proof of the fact that $j:X\to F$ (as above) must be constant. 
Let $x,y\in X$. Connect $x$ and $y$ by a smooth arc in $X$, and pick points
$a_0=x, b_0, a_1, b_1,\dots, a_n, b_n=y$ going along the arc so that we get
a fine partition of the arc. If every $b_{i-1}$ is very close to $a_i$, the
element 
$$
h=j(a_0)j(b_0)\obr j(a_1)j(b_1)\obr\dots j(a_n)j(b_n)\obr
$$
of the group $F$ is very close to $j(x)j(y)\obr$. On the other hand, we can find
$g\in G$ close to identity such that $g(a_i)$ is very close to $g(b_i)$
(\footnote{A neighborhood of our arc in $X$ can be identified with a Euclidean space so that
the arc corresponds to a straight line segment; in that case the existence of $g$
is geometrically obvious.}), and then
$gh=j(ga_0)j(gb_0)\obr j(ga_1)j(gb_1)\obr\dots j(ga_n)j(gb_n)\obr$ is close to $e_F$, 
the identity of $F$.
Thus the element $c=j(x)j(y)\obr$ of $F$ has the following property: we can find $h$ close
to $c$ and $g\in G$ close to $e_G$ so that $gh$ is close to $e_F$. It follows that $c=e_F$
and $j(x)=j(y)$.

2. If $G=\Diff X$ and $H\sbs G$ are as in Theorem~\ref{t:1}, the inclusion
$H\to G$ is {\em not} an epimorphism in the category \TopGr. Indeed, according to 
Theorem~\ref{t:2.1}, it suffices to construct a topological vector space $V$,
a jointly continuous linear representation of $G$ on $V$, and a non-constant $G$-equivariant map 
from $X$ to $V$. We can take for $V$ the space generated by points of $X$, that is, the space
of measures on $X$ with a finite support. Let us describe the topology on $V$ that suits our
purposes.

Let $V_0$ be the hyperplane in $V$ consisting of all measures of total mass zero. It will be sufficient
to describe a locally convex topology on $V_0$ such that the action of $G$ on $V_0$ 
is jointly continuous. Every metric
$d$ on $X$ gives rise to the ``transportation metric" (also known as the Kantorovich-Rubinstein metric, and under many other names, see \cite{villani}) $\bar d$ on $V_0$ defined by
$$
\bar d(v,0)=\inf\left\{\sum |c_i| d(x_i,y_i): v=\sum c_i(x_i-y_i)\right\},\ c_i\in \R, \ x_i, y_i\in X
$$
and $\bar d(u,v)=\bar d(u-v,0)$.
Consider a smooth Riemannian metric $d$ on $X$ and
the topology on $V_0$ generated by $\bar d$. 
If $g\in G$ is close to identity
in the $C^1$-topology, $g$ is a $C$-Lipschitz tranformation of $(X,d)$, where $C>1$ is a given
constant.
The joint continuity of the action of $G$ on $V_0$ easily follows (Proposition~\ref{p:Lipsh}).

Comparing this remark with the previous one, one can conclude that if the embedding $H\to G$ 
in the category \TopGr\ is an epimorphism, then the action of $G$ on $G/H$ is ``far from being smooth"
in a certain sense.

3. Epimorphisms in the the category {\bf AlgGr}
of linear algebraic groups over a given algebraically closed field
need not have a dense range \cite{Bor}
(we are grateful to Ugo Bruzzo for pointing out to us this reference as well as \cite{Brion}, and for his help with a copy of the article).
For example, if $G$ is a connected
linear algebraic group and $H$ is a parabolic subgroup of $G$, then the embedding
$H\to G$ is an epimorphism in {\bf AlgGr} (this readily follows from the fact that
every morphism of the projective variety $G/H$ to an affine variety is constant).
See \cite{Brion} for recent advances in this topic.

\section
{Idea of the proof of Main Theorem} 
\label{s:3}
Let $X$ be a compact connected
smooth
manifold, $G=\Diff X$, $H=\{f\in G:f(p)=p\}$ for some $p\in X$.
We want to prove the following:

\smallskip

{\it if $F$ is a locally convex Lie $G$-group and $j:X\to F$ is a
$G$-equivariant smooth mapping, then $j$ is constant.}

\smallskip

That will suffice, see Theorem~\ref{t:2.2} and the discussion after that.

Let $\frf$ be the Lie algebra
of $F$ \cite{M}, \cite{H}, \cite{Ne}. Then $G$ acts on $\frf$, and the action map
$G\ti \frf\to\frf$ is jointly continuous (in fact, even smooth, see Lemma~\ref{l:joint}). 
Equip the dual space $\frf^*$
with the topology of uniform convergence on compact sets. If $\frf$
is a Fr\'echet space, then $\frf^*$ is $\si$-compact (= the union of countably many compact
sets), see Lemma~\ref{l:4}. In general $\frf^*$ is covered by $G$-invariant $\sigma$-compact
subspaces (Corollary~\ref{c:1}). There is a natural 
mapping $j^*:\frf^*\to \Om^1(X)$, where $\Om^1(X)$ is the space 
of $C^\infty$-smooth differential 1-forms on $X$, equipped with the $C^\infty$-topology 
(= the topology of uniform convergence of all derivatives). The construction of $j^*$
is explained below. The mapping $j^*$ is continuous (Proposition~\ref{p:1}).
It is $G$-equivariant since $j$ is.
We shall show in the next section that
$\Omega^1(X)$ has no non-zero $G$-invariant $\sigma$-compact subspaces
(Proposition~\ref{p:2}). Since $\frf^*$ is covered by 
$G$-invariant $\sigma$-compact subspaces, it follows that $j^*=0$. 
This means that $j$ is constant.
%

If $M$ is a locally convex manifold and $V$ is a locally convex vector space, 
a {\em smooth $V$-valued differential 1-form} on  $M$ is a smooth mapping of the tangent
bundle of $M$ to $V$ which is linear on every $T_xM$, the tangent space to $M$ at $x\in M$.
The space $\frf^*$ can be identified with the space of 
all smooth left-invariant differential 1-forms on $F$, and the mapping $j^*:\frf^*\to \Om^1(X)$ 
from the previous paragraph is just the restriction of the natural map $\Omega^1(F)\to\Omega^1(X)$.

An alternative way to describe $j^*$ is the following.
Consider the {\em canonical} $\frf$-valued 1-form $\theta$ on $F$ such that $\theta(gv)=v$
for every $g\in F$ and $v\in \frf$, where $gv$ denotes the tangent vector to $F$ at $g$ obtained
from $v$ by the left translation by $g$. Consider the space $\Om^1(X, \frf)$
of all smooth $\frf$-valued 1-forms on $X$. Let $\tau\in \Omega^1(X,\frf)$ be the inverse image
of $\theta$ under $j$. In other words, if $u$ is a tangent vector to $X$ at $x$, then 
$\tau(u)=\theta(j_*(u))$, where $j_*$ is the differential of $j$ at $x$. 
Every $h\in \frf^*$ induces a natural map 
$h_*:\Om^1(X, \frf)\to \Om^1(X, \R)=\Om^1(X)$, and we have $j^*(h)=h_*(\tau)$. 
Indeed, both forms have equal values on any vector $u$ tangent to $X$, namely, 
the value that equals $h(\theta(j_*(u)))$.

\section
{Details of the proof}
\label{s:4}
It remains to prove the statements mentioned in Section~3.

\begin{lemma}
\label{l:joint}
If $G$ and $F$ are locally convex groups and $G$ acts smoothly on $F$ by automorphisms, 
the the action of $G$ on the Lie algebra of $F$ is jointly continuous (in fact, smooth).
\end{lemma}

\begin{proof}
The action of $G$ on the Lie algebra $\mathfrak f$ on $F$ is exactly the restriction to $G\times {\mathfrak f}$ of the adjoint representation of the semidirect product $G\ltimes F$, viewed as a map $(G\ltimes F)\times ({\mathfrak g}\ltimes {\mathfrak f})\to {\mathfrak g}\ltimes {\mathfrak f}$. But the adjoint representation is a smooth mapping (\cite{Ne}, p. 333, Ex. II.3.9).
\end{proof}

\begin{lemma}
\label{l:4}
If $E$ is a metrizable LCS, then the dual space $E^*$, equipped with the topology
of uniform convergence on compact sets, is $\sigma$-compact.
\end{lemma}
This is of course well known. We decided to include a proof, since this is a crucial
point in our arguments.

\begin{proof}
Pick a countable base $(U_n)$ of neighborhoods of zero in $E$. Then $E^*=\bigcup U_n^\circ$, 
where $U^\circ$ is the polar set, $U^\circ=\{f\in E^*:\sup\{|f(x)|:x\in U\}\le1\}$. Every 
$U_n^\circ$ is compact with respect to the $w^*$-topology (the Banach--Alaoglu Theorem).
Since $U_n^\circ$ is equicontinuous, the $w^*$-topology
on $U_n^\circ$ is the same as the topology of uniform convergence on compact sets
\cite[Ch.~3, page~17, Proposition 5]{BourbEVT}.
\end{proof}

\begin{proposition}
\label{p:1}
Let $X$ be a compact smooth manifold, $F$ a locally convex Lie group, 
$\frf$ the Lie algebra of $F$, 
$j:X\to F$ a smooth mapping.
The mapping $j^*:\frf^*\to \Om^1(X)$ considered in Section~3 is continuous.
(The topology on $\frf^*$ is the topology of uniform convergence on compact 
subsets of $\frf$.)
\end{proposition}

\begin{proof}
Pick a Riemannian metric on $X$.
For a smooth vector field $v$ on $X$ denote by $L_v$ the Lie differentiation
with respect to $v$ \cite[Ch.1, \S 3]{KN}.
A typical neighborhood of zero in $\Om^1(X)$ looks as follows: pick smooth vector
fields $v_1,\dots, v_k$ on $X$, and consider those forms $\om\in \Om^1(X)$ for which
the form $\eta=L_{v_1}\dots L_{v_k}\om$ has a small norm,
that is, $|\eta(u)|<\e$ for all unit tangent vectors $u$.

Suppose vector fields $v_1,\dots, v_k$ and $\e>0$ are given. We must prove the following:
if $h\in \frf^*$ is small on a certain compact set $K$ and $\om=j^*(h)$, 
then $|L_{v_1}\dots L_{v_k}\om(u)|<\e$
for all unit tangent vectors $u$. Consider the space $\Om^1(X, \frf)$
of all smooth $\frf$-valued 1-forms on $X$. 
If $v$ is a smooth vector field on $X$,
the Lie differentiation $L_v$ is well-defined as a map from $\Om^1(X, \frf)$ to itself.

Consider the forms $\theta\in \Omega^1(F, \frf)$ and $\tau\in\Om^1(X, \frf)$
introduced in the previous section:
$\theta$ is the canonical $\frf$-valued left-invariant 1-form on $F$, and
$\tau=j^*(\theta)$. 
Denote by $K$
the compact subset of $\frf$ consisting of all vectors of the form 
$L_{v_1}\dots L_{v_k}\tau(u)$, where $u$ runs over all unit tangent vectors to $X$.
If $h$ is $\e$-small on $K$, $\om=j^*(h)=h_*(\tau)$, and $u$ is a unit tangent vector to $X$,
then $L_{v_1}\dots L_{v_k}\om(u)=h_*(L_{v_1}\dots L_{v_k}\tau)(u)$, since $h_*$ commutes with 
every Lie differentiation $L_v$. The last expression belongs to $h(K)$ and hence is $\e$-small.
Thus $j^*$ is continuous.
\end{proof}

A topological space is {\em Polish} if it is homeomorphic to a complete separable
metric space. A topological space is Polish if and only if it is \v{C}ech-complete
and has a countable base (see \cite{E}, Th. 4.3.26). If a topological group with a countable base is locally 
\v{C}ech-complete (that is, admits a \v{C}ech-complete neighborhood of the identity),
then it is \v{C}ech-complete (\cite{AT}, Proposition 4.3.17), hence Polish. In particular, every FL-group with a countable base is Polish.

Recall that a set in a topological space is {\it meagre\/} if it is
contained in the union of countably many nowhere dense sets. If a space $Y$
is completely metrizable (that is, admits a compatible complete metric),
then no meagre subset of $Y$ contains interior points, in virtue of
the Baire category theorem.

\begin{lemma}
\label{l:2}
Suppose that a topological group $G$ acts continuously on a Hausdorff space
$Y$. Suppose that for every neighbourhood $U$ of unity in $G$ and for
every $y\in Y$ the closure of the set $Uy$ in $Y$ is not compact. If $G$ is
not meagre in itself (in particular, if $G$ is Polish),
then $Y$ contains no non-empty $G$-invariant $\si$-compact subsets.
\end{lemma}

\begin{proof}
Let $y\in A=\bigcup_{n=0}^\infty K_n\sbs Y$, where each $K_n$ is compact.
We must show that the orbit $Gy$ is not contained in $A$. Let
$F_n=\{g\in G: gy\in K_n\}$. Then each $F_n$ is closed in $G$ and has no
interior points. Since $G$ is not meagre, there exists a $g\in G\setminus
\bigcup_n F_n$, and we have $gy\notin A$.
\end{proof}

\begin{proposition}
\label{p:2}
If $X$ is a compact manifold and $G=\Diff X$, the space $\Om^1(X)$ contains
no non-zero $G$-invariant $\si$-compact subspaces.
\end{proposition}

\begin{proof} (\footnote{The referee pointed out that the tools of \cite{Dedi}
can be used to prove this proposition.})
Apply Lemma~\ref{l:2} to the group $G=\Diff X$
and the space $Y=\Om^1(X)\setminus\{0\}$. Since $G$ is Polish, 
it suffices to check that for any neighbourhood $U$
of the identity in $G$, any non-zero 1-form $\om\in\Om^1=\Om^1(X)$
and any compact subset $K\sbs\Om^1$ we have $U\om\not\sbs K$.
We'll confine ourselves to the case when $X$ is a circle. Let $\om=f(\theta)\,d\theta$.
Assume that $U$ consists of all diffeomorphisms $g$ of $X$ such that the first $k$ derivatives
of $g$ are close to those of the identity map. As the $(k+1)$-th derivative of $g\in U$
can be arbitrary large, the differential forms $g^*(\om)=f(g(\theta))g'(\theta)\,d\theta$,
where $g$ runs over $U$,
cannot all lie in a compact set.
\end{proof}

If $G$ is a Polish group that acts on an LCS $V$ by linear transformations so that
the action law $G\times V\to V$ is jointly continuous, we say that $V$ is a $G$-module.

\begin{proposition}
\label{p:factor}
If $G$ is a Polish group and $V$ is a $G$-module, the topology of $V$ is generated by
$G$-morphisms of $V$ to metrizable $G$-modules. 
\end{proposition}

In more detail, the assertion is that for every neighborhood $U$ of zero in $V$ there exist
a $G$-morphism $p:V\to V'$ to a metrizable $G$-module $V'$ and a neighborhood $U'$ of zero in $V'$
such that $p\obr(U')\sbs U$.

\begin{proof}
Let $U$ be a given neighborhood of zero in $V$. For every $g\in G$ find a neighborhood $W_g$ of $g$
in $G$ and a neighborhood $U_g$ of zero in $V$ so that $W_gU_g\sbs U$. There is a countable collection
of $W_g$'s that cover $G$; consider the corresponding $U_g$'s. With each of these $U_g$ do the same
as we did with $U$. We get a larger countable collection of neighborhoods of zero. Proceed in a similar
manner. In countably many steps we obtain a countable collection $\sB$ of neighborhoods of zero such that for every $O\in \sB$ and $g\in G$ there exist $O'\in \sB$ and a neighborhood $W$ of $g$ in $G$
such that $WO'\sbs O$. We can throw in finite intersections and images under homotheties 
with coefficient $1/2$. We assume that all neighborhoods in $\sB$ are convex and symmetric. In this way
we can achieve that $\sB$ is a filter base, and the filter $\sF$ generated by $\sB$ is the filter of neighborhoods for some pseudometrizable locally convex topology $\sT$ on $V$ that is coarser that 
the original topology $\sT_0$. By our construction, the map $G\times V\to V$ is continuous with respect to $\sT$ at every point of the form $(g,0)$. We claim that it is continuous everywhere.
Indeed, let $(g,x)\in G\ti V$ and $O\in \sB$. Pick a neighborhood $W$ of $g$ and $O'\in \sB$ so that 
$WO'\sbs O$. Shrinking $W$ if necessary, we may assume that $Wx\sbs x+O$ (we just use the 
$\sT_0$-continuity of the action). Then $W(x+O')\sbs Wx+WO'\sbs x+O+O$, and the $\sT$-continuity
of the action follows. We take for $V'$ the metrizable space associated with the pseudometrizable space
$(V,\sT)$.
\end{proof}

\begin{corollary}
\label{c:1}
Suppose $G$ is a Polish group and $V$ is a $G$-module. Equip $V^*$ with the topology of uniform
convergence on compact sets. The space $V^*$ is covered by $G$-invariant $\sigma$-compact subspaces
(which may be non-closed). 
\end{corollary}

\begin{proof}
If $p:V\to W$ is a $G$-morphism to a metrizable $G$-module $W$, the image of $p^*:W^*\to V^*$
is $\sigma$-compact (Lemma~\ref{l:4}) and $G$-invariant. 
According to Proposition~\ref{p:factor}, $W^*$ is covered by subspaces of this sort.
\end{proof}

\section{The case of finite dimensional Lie groups}
\label{s:fdLie}

\begin{theorem}
Let $f\colon G\to H$ be a Lie group morphism between two finite dimensional Lie groups. Assume that $f$ is an epimorphism in the category of locally convex Lie groups 
(regular Fr\'echet--Lie groups, if $G$ has a countable base). 
Then $f$ has dense range.
\end{theorem}

\begin{proof}
The Lie group $H$ acts smoothly by left translations on the space $C^{\infty}(H)$ equipped with the topology of compact convergence with all derivatives (\cite{Ne2}, proof of Prop. 4.6). That is, the action mapping $\lambda\colon H\times C^{\infty}(H)\to C^{\infty}(H)$, $\lambda_g(f)(x)=f(g^{-1}x)$, is smooth. The complete locally convex space $C^{\infty}(H)$ is an abelian Lie group. 

The morphism $f$ induces a homomorphism between discrete groups $f_0\colon G/G_0\to H/H_0$, where $G_0$ and $H_0$ denote the connected components. This $f_0$ is in particular an epimorphism in the category of discrete groups (Lie groups of dimension zero), so is onto by an argument of Kurosh, Livshits and Shul'geifer, see \cite{Nu}, top of p. 156
(or Theorem 4.3.1 in \cite{P2}).
We conclude: the right homogeneous space $\overline{f(G)}\backslash H$ is connected. 
Find a smooth function $h_1\colon \overline{f(G)}\backslash H\to\R$ achieving its maximum 
exactly on the right coset $\overline{f(G)}$. The composition of $h_1$ with the right quotient 
map $H\to \overline{f(G)}\backslash H$ gives an element $h$ of $C^{\infty}(H)$ whose stabilizer 
under the action $\lambda$ is exactly $\overline{f(G)}$. According to Theorem~\ref{t:2.2}
(equivalence of (1) and (3)), $h$ is $H$-invariant, that is, constant. It follows that
$\overline{f(G)}=H$.

If $G$ is has a countable base, that is, has at most countably many connected components, 
the above argument (using only discrete groups, which are Fr\'echet--Lie, indeed zero-dimensional Lie) shows that $H$ has the same property,
therefore $C^{\infty}(H)$ is a Fr\'echet space.
\end{proof}

Note that in the category of connected finite-dimensional Lie groups epimorphisms need not have
a dense range. The existence of epimorphically embedded complex algebraic groups, like
the subgroup of upper triangular matrices in $GL(n,C)$ \cite{Bor}, implies the analogous fact
for finite-dimensional complex Lie algebras, and the case of
finite-dimensional real Lie
algebras follows by complexification. See the next section (Thm. \ref{t:utsl2}) for a stronger result: some of those embeddings are epimorphisms even in the category of Banach--Lie groups.

\section {The case of Banach--Lie groups}
\label{s:BL}

We consider {\em Banach--Lie} groups, or BL-groups, as defined in \cite{BLie} under the name
of (real) Lie groups. We are going to prove the following: if a morphism $f:G\to H$ between BL-groups
is a an epimorphism in \TopGr, then $f$ has a dense range. More generally:

\begin{theorem}
\label{t:BL}
Let $G$ be a BL group, $H$ a proper closed subgroup. Then the inclusion $H\to G$
is not an epimorphism in \TopGr.
\end{theorem}

The argument is the same as in Remark 2, Section~\ref{s:2}. We consider a certain metric $d$
on the space $P=G/H$ and its extension $\bar d$ over the space $V_0$ of measures on $P$
with a finite support and total mass zero. We check that the action of $G$ on $V_0$ is jointly 
continuous. If $V=V_0\oplus\R$ is the space of all measures on $P$ with a finite support, 
there is a natural $G$-map from $P$ to $V$, and we invoke Theorem~\ref{t:2.1} (equivalence of (1) and
(4)) to conlude that the inclusion $H\to G$ is not an epimorphism in \TopGr.

\begin{proposition}
\label{p:Lipsh}
Suppose a topological group $G$ continuously acts on a metric space $(M,d)$ by Lipschitz
transformations. Suppose there is a neighborhood $U$ of the unity in $G$ and a constant $C>0$
such that for every $g\in U$ the $g$-shift  $\si_g:M\to M$ is $C$-Lipschitz. 
Let $V$ be the space of measures
on $M$ with a finite support, $V_0$ the hyperplane of measures of total mass zero. Equip $V_0$
with the Kantorovich--Rubinstein metric $\bar d$. 
Then the action of $G$ on $V_0$ is jointly continuous.
\end{proposition}

This is essentially a version of Megrelishvili's result \cite[Theorem 4.4]{Me}.
If $G$ is generated by $U$ (this happens, for example, if $G$ is connected), we can drop
the assumption that $G$ acts by Lipschitz transformations, as this will follow from the condition that
$\si_g$ is Lipschitz for every $g\in U$.

\begin{proof}
If $\bar d(v,0)<\e$, write $v=\sum c_i(x_i-y_i)$ so that $\sum |c_i| d(x_i,y_i)<\e$.
If $g\in U$, we have $\bar d(gv,0)\le\sum |c_i| d(gx_i,gy_i)<C\e$. It follows that the action
$G\ti V_0\to V_0$ is jointly continuous at $(1_G, 0)$. Similarly, the Lipschitz condition 
for $\si_g$ implies that the action is separately continuous. We conclude that the action
is jointly continuous at every point $(g_0, v_0)$: if $h$ is close to $1_G$ and $v\in V_0$ is small, 
we have $hg_0(v_0+v)-g_0v_0=(hg_0v_0-g_0v_0)+hg_0v$, where both summands are small.
\end{proof}

A {\em norm} on a group $G$ is a function $p:G\to \R_+$ such that: 
(1) $p(1_G)=0$ and $p(x)>0$ if $x\ne 1_G$;
(2) $p(x)=p(x\obr)$; 
(3) $p(xy)\le p(x)+p(y)$. If $G$ is a metrizable group, its topology is
generated by a certain norm $p$, in the sense that the sets of the form $\{x\in G:p(x)<\e\}$
constitute a base at $1_G$. (Note that not every norm generates a group topology.) If $H$ is a closed
subgroup of $G$, we can define a compatible metric $d$ on $G/H$ by 
$$
d(a,b)=\inf\{p(g): ga=b\}.
$$

We now prove Theorem~\ref{t:BL}. 
First we consider the case when $G$ is a connected BL-group. 
Denote by $\frg$ its Lie algebra. We assume that a norm on $\frg$
is given such that $\frg$ is a Banach space and $\norm{[X,Y]}\le \norm{X}\cdot\norm{Y}$.
The {\em exponential length} norm on $G$ is defined by
$$
p(g)=\inf\{\norm{X_1}+\dots\norm{X_n}: g=e^{X_1}\dots e^{X_n}\}.
$$
This norm is compatible with the topology of $G$ \cite[Proposition 3.2]{ADM}.
Note that we may assume that all vectors $X_i$ in the definition above are short, that is, 
belong to a given neighborhood of zero in $\frg$: otherwise replace a vector $X$ by $k$ vectors
$X/k$. This will not change the sum of norms or the product $e^{X_1}\dots e^{X_n}$.

Equip $G/H$ with the corresponding metric $d$ defined as above:
\[
d(a,b)=\inf\{\norm{X_1}+\dots\norm{X_n}: e^{X_1}\dots e^{X_n}a=b\}.
\]
Pick a small $\e$, and let $U=\{e^X:\norm{X}<\e\}\sbs G$ be the corresponding neighborhood
of $1_G$. 
If $\norm{Y}<\e$, $g=e^Y\in U$ and $a,b\in P$, then 
\begin{align}
\label{f1}
d(ga,gb)&=\inf\{\norm{X_1}+\dots\norm{X_n}: e^{X_1}\dots e^{X_n}ga=gb\}\\
&=
\inf\{\norm{X_1}+\dots\norm{X_n}: e^{-Y}e^{X_1}\dots e^{X_n}e^Ya=b\}.
\nonumber
\end{align}
Pick short vectors $X_1',\dots,X_n'$ so that 
\[
e^{X'_1}\dots e^{X'_n}a=b
\]
and
\begin{equation}
\norm{X'_1}+\dots+\norm{X'_n}< 2d(a,b).
\label{eq:2dab}
\end{equation}
Pick short vectors $X_i$ so that $e^{X_i}=e^{Y}e^{X_i'}e^{-Y}$.
Then $e^{-Y}e^{X_1}\dots e^{X_n}e^Ya=b$, and hence, according to the formula (\ref{f1}),
\begin{equation}
d(ga,gb)\le \norm{X_1}+\dots\norm{X_n}.
\label{eq:dgagb}
\end{equation}
We have $X_i=H(Y, H(X'_i,-Y))=X_i'+[Y, X_i']+$ terms of higher degree, where $H$
is the Hausdorff series \cite[Ch.2, \S 6,7]{BLie}. If $\e$ is small enough, we have
$\norm{X_i}\le 2\norm{X_i'}$, therefore, from Eq. (\ref{eq:2dab}) and (\ref{eq:dgagb}), 
$d(ga,gb)<4d(a,b)$. 

By Proposition~\ref{p:Lipsh},
the action of $G$ on $V$ is jointly continuous. This proves Theorem~\ref{t:BL}
in the case $G$ is connected.

If we drop the assumption that $G$ is connected, the argument
with connected components used in the previous section shows that we can assume
that $H$ meets all connected components of $G$. In that case $G/H=G_0/(G_0\cap H)$,
and we have seen above that there is metric on the manifold $G/H$ such that $G_0$ acts
on this manifold by Lipschitz transformations. Actually, the whole group $G$ acts by Lipschitz
transformations. This can be deduced from the following observation: every automorphism
of the connected Lie group $G_0$ is Lipschitz, if $G$ is equipped with the exponential length
metric as above. To see that the latter statement is true, note that the tangent automorphism
$\si$ of the Lie algebra $\frg$ is such that 
$C\obr\norm{X}\le\norm{\si X}\le C\norm{X}$ for some constant $C>0$.

An early version of the present paper contained the following questions: do epimorphisms $f\colon G\to H$
in the category of BL-groups have a dense range, in particular if $G$ and $H$ are (connected) finite-dimensional Lie groups? The anonymous referee provided the following
counterexample to all three questions.

\begin{theorem}
Let $B$ be the subgroup of
upper-triangular matrices in $SL_2(\R)$. Then the embedding $B\to SL_2(\R)$ is an epimorphism
in the category of BL-groups. 
\label{t:utsl2}
\end{theorem}

\begin{proof}
It suffices to prove the analogous fact for Lie algebras which reduces
to the following. Let $\frg$ be a Banach--Lie algebra. Consider $\frsl_2$-triples $(h,e,f_1)$ and
$(h,e,f_2)$ in $\frg$, that is, 
$$
[h,e]=2e,\quad [h,f_j]=-2f_j,\quad [e,f_j]=h.
$$
We must prove that $f_1=f_2$. Assume the contrary: $f=f_1-f_2\ne0$. Put 
$\frs=\mbox{span\,}\{e,h, f_1\}$, $v_n=(\mbox{ad\,}f_1)^nf$,
and $V=\mbox{span\,}\{v_0, v_1, \dots\}$.
Since $[e,f]=0$ and $[h,f]=-2f$, it is easy to see that $V$ is an $\frs$-module
such that each $v_n$ is an $\mbox{ad\,}h$-eigenvector with the eigenvalue $-2-2n$.
As $\mbox{ad\,}h$ is a bounded
operator, the vectors $v_n$ are zero for $n$ large enough. We therefore obtain a 
finite-dimensional representation of $\frsl_2$ in which $h$ has negative spectrum. 
That is impossible:
the spectrum of $h$ in a finite-dimensional representation of $\frsl_2$ is 
symmetric with respect to zero \cite[Ch.8, \S 1]{B8}.
\end{proof}

\section{Open questions}

Let $f\colon G\to H$ be two groups from a category marking a row, and suppose $f$ is an epimorphism in a category marking a column. Does $f$ necessarily have a dense range?

Here is a summary of what we know. 

\vskip .4cm

\begin{center}
\begin{tabular}{|l||c|c|c|c|c|c|}
\hline
& connected & \multirow{2}{*}{f.-d. Lie} & \multirow{2}{*}{B-Lie} & regular & \multirow{2}{*}{LC Lie} & \multirow{2}{*}{top.}  \\
& f.d. Lie & & &  F-Lie & & \\
\hline \hline
connected f.-d. Lie & $\times$ & $\times$ &$\times$ & $\checkmark$ & $\checkmark$  & $\checkmark$  \\
\hline
f.-d. Lie & & $\times$ & $\times$ & ? & $\checkmark$  & $\checkmark$  \\
\cline{1-1}
\cline{3-7}
Banach-Lie & \multicolumn{2}{c|}{} & $\times$ & ? & ? & \checkmark \\
\cline{1-1}
\cline{4-7}
Fr\'echet-Lie & \multicolumn{3}{c|}{} & $\times$ & $\times$ & ? \\
\cline{1-1}
\cline{5-7}
loc. conv. Lie & \multicolumn{4}{c|}{} & $\times$  & ? \\
\cline{1-1}
\cline{6-7}
top. groups & \multicolumn{5}{c|}{} & $\times$\\ 
\cline{1-1}
\cline{7-7}
\end{tabular}
\end{center}
\vskip .4cm

As we mentioned in Section~\ref{s:BL}, 
the three crosses appearing in the BL-column 
are due to the anonymous referee. We cordially thank the referee for this valuable contribution. 

We have seen that whether or not the inclusion $i:H\to G$ of a proper
closed subgroup is an epimorphism in a suitable category depends on the dynamical system $(G, G/H)$. For example, if it is smooth (or, more generally, Lipschitz, see Proposition \ref{p:Lipsh}), the inclusion $i$ is not an epimorphism of Hausdorff groups. In order for $i$ to be an epimorphism in the category \TopGr, the action of $G$ on $G/H$ must be ``sufficiently mixing''. Formalizing this criterion could be an interesting task.




\begin{thebibliography}{99}
 

\bibitem{ADM} H. Ando, M. Doucha, Y. Matsuzawa, \newblock {\em Large scale geometry of Banach--Lie groups,} \newblock  arXiv:2011.10376v2 [math.OA], version of Dec. 11, 2020.

\bibitem{AT} A. Arhangel'skii and M. Tkachenko, \newblock {\em 
Topological groups and related structures,}
\newblock Atlantis Studies in Mathematics, \textbf{1}, Atlantis Press, Paris; World Scientific Publishing, Hackensack, NJ, 2008.

\bibitem{Bor} F. Bien, A. Borel, \newblock {\em Sous-groupes \'epimorphiques de groupes alg\'ebriques
lin\'eaires I,II,} \newblock C. R. Acad. Sci. Paris S\'er. I, {\bf 315} (1992), 649--653, 1341--1346.

\bibitem{Brion} M. Brion, \newblock {\em Epimorphic subgroups of algebraic groups,}
\newblock Math. Res. Lett. \textbf{24} (2017), no. 6, 1649--1665.

\bibitem{BourbEVT} N.~Bourbaki, \newblock {\em Espaces vectoriels topologiques,} \newblock Springer, 2007.

\bibitem{BLie} N.~Bourbaki, \newblock {\em Lie groups and Lie algebras. Chapters 1--3}, \newblock Springer, 1989.

\bibitem{B8} N.~Bourbaki, \newblock {\em Lie groups and Lie algebras. Chapters 7-9}, \newblock Springer, 2005.

\bibitem{Dedi} H. Gl\"ockner, \newblock {\em Solutions to open problems in Neeb's recent survey on infinite-
dimensional Lie groups,} \newblock Geom. Dedicata {\bf 135} (2008), 71-86.

\bibitem{E} R. Engelking, \newblock {\em General topology,} Second edition. \newblock Sigma Series in Pure Mathematics, \textbf{6}, Heldermann Verlag, Berlin, 1989.

\bibitem{H} R.S. Hamilton,
\newblock {\em The inverse function theorem of Nash and Moser,}
\newblock Bull. Amer. Math. Soc {\bf 7} (1982), No. 1, 65--222.

\bibitem{KN} S.~Kobayashi, K.~Nomizu, \newblock {\em Foundations of Differential Geometry. Vol.~1}, 
\newblock Interscience Publishers, 1963.

\bibitem{Me} M.G. Megrelishvili, \newblock {\em Free topological G-groups,}
\newblock New Zealand J. Math. \textbf{25} (1996), 59--72.

\bibitem{M} J. Milnor, 
\newblock {\em Remarks on infinite-dimensional Lie groups,} \newblock in book: 
Relativit\'e, groupes et topologie II. North-Holland, 1984. Pp. 1007--1057.

\bibitem{Ne} K.-H. Neeb, \newblock {\em Towards a Lie theory of locally convex groups,}
\newblock Jpn. J. Math. \textbf{1} (2006), 291--468.

\bibitem{Ne2} K.-H. Neeb, \newblock {\em On differentiable vectors for representations of infinite dimensional Lie groups,} \newblock J. Funct. Anal. \textbf{259} (2010), 2814--2855. 

\bibitem{Nu} E.C. Nummela, \newblock {\em On epimorphisms of topological groups,}
\newblock General Topology Appl. \textbf{9} (1978), 155--167.

\bibitem{O} H. Omori, \newblock {\em Infinite-Dimensional Lie Groups,} \newblock Translations of Mathematical Monographs, \textbf{158}, American Mathematical Society,
1997.

\bibitem{Pe} V. Pestov, \newblock {\em Epimorphisms of Hausdorff groups by way of topological dynamics,} \newblock New Zealand J. Math. \textbf{26} (1997), 257--262.

\bibitem{P2} V. Pestov, \newblock {\em Topological groups: Where to from here?}, 
\newblock Topology Proceedings \textbf{24} (1999), 421--502; arXiv:math/9910144.

\bibitem{U1} V. Uspenskij, 
\newblock {\em The solution of the epimorphism problem for
Hausdorff topological groups,} \newblock Seminar Sophus Lie {\bf 3} (1993), 69--70.

\bibitem{U2} V. Uspenskij, 
\newblock {\em The epimorphism problem for Hausdorff topological groups,}
\newblock Topology Appl. {\bf 57} (1994), 287--294.

\bibitem{U3} V. Uspenskij, 
\newblock {\em Epimorphisms of topological groups and $Z$-sets
in the Hilbert cube,} 
\newblock Symposia Gaussiana. Proceedings of the 2nd Gauss Symposium
(Munich,  
1993). Conference A: Mathematics
and Theoretical Physics (edited by M. Behara, R. Fritsch, R.G. Lintz).
Walter de Gruyter \& Co., Berlin -- New York, 1995. Pp. 733-738.

\bibitem{villani}  C. Villani, \newblock {\em Optimal transport. Old and new,} \newblock Grundlehren der Mathematischen Wissenschaften \textbf{338}, Springer-Verlag, Berlin, 2009. 


\end{thebibliography}
\end{document}